\documentclass{amsart}
\usepackage{amsmath,amsfonts,amssymb,enumerate,mathrsfs,mathtools}

\newcommand{\N}{\mathbb{N}}                   
\newcommand{\R}{\mathbb{R}}                   
\newcommand{\A}{\mathscr{A}_a}                
\newcommand{\AR}{\mathscr{A{\!}R}}            
\newcommand{\Zar}{\mathrm{Zar}}
\newcommand{\Na}{\mathcal{N}}                 
\newcommand{\Reg}{\mathrm{Reg}}
\newcommand{\Sing}{\mathrm{Sing}}
\newcommand{\I}{\mathcal{I}}

\theoremstyle{plain}
\newtheorem{theorem}{Theorem}
\newtheorem{proposition}{Proposition}
\newtheorem{lemma}{Lemma}
\newtheorem{corollary}{Corollary}

\theoremstyle{definition}

\newtheorem{remark}{Remark}

\newtheorem{problem}{Problem}

\numberwithin{equation}{section}

\begin{document}

\title{Extensions of arc-analytic functions}

\author{Janusz Adamus}
\address{Department of Mathematics, The University of Western Ontario, London, Ontario, Canada N6A 5B7}
\email{jadamus@uwo.ca}
\author{Hadi Seyedinejad}
\address{Department of Mathematics, The University of Western Ontario, London, Ontario, Canada N6A 5B7}
\email{sseyedin@uwo.ca}
\thanks{J. Adamus's research was partially supported by the Natural Sciences and Engineering Research Council of Canada}


\begin{abstract}
We prove that every arc-analytic semialgebraic function on an arc-symmetric set $X$ in $\R^n$ admits an arc-analytic semialgebraic extension to the whole $\R^n$.
\end{abstract}
\maketitle


\section{Introduction}
\label{sec:intro}

Arc-analytic functions play an important role in modern real algebraic and analytic geometry (see, e.g., \cite{KuPa1} and the references therein). They are, however, hardly known outside the specialist circles, which is perhaps partly due to their rather surprising, if not pathological, behaviour in the general analytic setting (see \cite{BMP}). In the algebraic setting, on the other hand, arc-analytic functions form a very nice family, as our main result will hopefully contribute to attesting to.

Let us recall that a function $f:X\to\R$ is called \emph{arc-analytic} when $f\circ\gamma$ is an analytic function for every real analytic arc $\gamma:(-1,1)\to X$. Typically, in the literature, $X$ is assumed to be a smooth real algebraic or analytic variety, or a semialgebraic set.
\medskip

In this article, we are interested in semialgebraic arc-analytic functions in the setting in which they were originally introduced by Kurdyka \cite{Ku}, that is, on arc-symmetric semialgebraic sets. Recall that a \emph{semialgebraic} set in $\R^n$ is one that can be written as a finite union of sets of the form $\{x\in\R^n:p(x)=0,q_1(x) >0,\dots,q_r(x)>0\}$, where $r\in\N$ and $p,q_1,\dots,q_r\in\R[x_1,\dots,x_n]$. A semialgebraic set $X\subset\R^n$ is called \emph{arc-symmetric} if, for every analytic arc $\gamma:(-1,1)\to\R^n$ with $\gamma((-1,0))\subset X$, we have $\gamma((-1,1))\subset X$.
 A function $f:X\to\R$ is a \emph{semialgebraic function} when its graph is a semialgebraic subset of $\R^{n+1}$. Every arc-analytic semialgebraic function on an arc-symmetric set is continuous in the Euclidean topology (\cite[Prop.\,5.1]{Ku}).

By a fundamental theorem \cite[Thm.\,1.4]{Ku}, the arc-symmetric semialgebraic sets are precisely the closed sets of a certain noetherian topology on $\R^n$. (A topology is called \emph{noetherian} when every descending sequence of its closed sets is stationary.) Following \cite{Ku}, we will call it the \emph{$\AR$ topology}, and the arc-symmetric semialgebraic sets will henceforth be called \emph{$\AR$-closed sets}.

Given an $\AR$-closed set $X$ in $\R^n$, we denote by $\A(X)$ the ring of arc-analytic semialgebraic functions on $X$.
The elements of $\A(X)$ play the role of `regular functions' in $\AR$ geometry. Indeed, it is not difficult to see (\cite[Prop.\,5.1]{Ku}) that the zero locus of every arc-analytic semialgebraic function $f:X\to\R$ is $\AR$-closed. Recently, it was also shown (\cite[\S\,1, Thm.\,1]{AS}) that every $\AR$-closed set may be realized as the zero locus of an arc-analytic function. Therefore, the $\AR$ topology is, in fact, the one defined by arc-analytic semialgebraic functions.
\medskip

In \cite{AS}, we conjectured that every arc-analytic semialgebraic function on an $\AR$-closed set $X$ in $\R^n$ is a restriction of an element of $\A(\R^n)$. Theorem~\ref{thm:main} gives an affirmative answer to this conjecture.

If $X$ is an $\AR$-closed sets in $\R^n$, we denote by $\I(X)$ the ideal in $\A(\R^n)$ of the functions that vanish on $X$.

\begin{theorem}
\label{thm:main}
Let $X$ be an $\AR$-closed set in $\R^n$, and let $f:X\to\R$ be an arc-analytic semialgebraic function. Then, there exists an arc-analytic semialgebraic $F:\R^n\to\R$ such that $F|_X=f$. In other words,
\[
\A(X)\simeq\A(\R^n)/\I(X)
\]
as $\R$-algebras.
\end{theorem}

\begin{remark}
\label{rem:regulous}
The above theorem seems particularly interesting in the context of continuous rational functions.
Following \cite{KKK}, we will call $f:X\to\R$ a \emph{continuous rational function} when $f$ is continuous (in the Euclidean topology) and there exist a Zariski open dense subset $Y$ in the Zariski closure $\overline{X}^\Zar$ and a regular function $F:Y\to\R$ such that $f|_{X\cap Y}=F|_{X\cap Y}$. 
Continuous rational functions have been extensively studied recently (see, e.g., \cite{Kuch}, \cite{KN}, \cite{FHMM}, \cite{KKK}).

It follows from the proof of \cite[Thm.\,1.12]{KKK} (which works also in the $\AR$ setting) that every continuous rational function on an $\AR$-closed set $X\subset\R^n$ is an element of $\A(X)$, and hence admits an arc-analytic semialgebraic extension to $\R^n$.
However, in general, a continuous rational function on $X$ cannot be extended to a continuous rational function on $\R^n$, even if $X$ is Zariski closed (see \cite[Ex.\,2]{KN}). To overcome this problem, Koll{\'a}r and Nowak introduced the notion of a \emph{hereditarily rational function}, that is, a continuous function on an algebraic set which remains rational after restriction to an arbitrary algebraic subset (see \cite{KN} for details). The main result of \cite{KN} asserts that a function $f:Z\to\R$ on an algebraic set $Z\subset\R^n$ is hereditarily rational if and only if $f$ admits a continuous rational extension to $\R^n$.
\end{remark}
\medskip

We shall prove Theorem~\ref{thm:main} in Section~\ref{sec:proof}. We show some immediate corollaries of our main result in Section~\ref{sec:corollaries}. For the reader's convenience, in Section~\ref{sec:prelim}, we recall basic notions and tools used in this note.


\section{Preliminaries}
\label{sec:prelim}

\subsection{$\AR$-closed sets}

First, we shall recall several properties of $\AR$-closed sets that will be used throughout the paper. For details and proofs we refer the reader to \cite{Ku}.

The class of $\AR$-closed sets includes, in particular, the algebraic sets as well as the Nash sets (see below). The $\AR$ topology is strictly finer than the Zariski topology on $\R^n$ (see, e.g., \cite[Ex.\,1.2]{Ku}). Moreover, it follows from the semialgebraic Curve Selection Lemma that $\AR$-closed sets are closed in the Euclidean topology on $\R^n$ (see \cite[Rem.\,1.3]{Ku}).

An $\AR$-closed set $X$ is called \emph{$\AR$-irreducible} if it cannot be written as a union of two proper $\AR$-closed subsets. It follows from noetherianity of the $\AR$ topology (\cite[Prop.\,2.2]{Ku}) that every $\AR$-closed set admits a unique decomposition $X=X_1\cup\dots\cup X_r$ into $\AR$-irreducible sets satisfying $X_i\not\subset\bigcup_{j\neq i}X_j$ for each $i=1,\dots,r$. The sets $X_1,\dots,X_r$ are called the \emph{$\AR$-components} of $X$.

Noetherianity of the $\AR$-topology implies as well that an arbitrary family of $\AR$-closed sets has a well defined intersection. In particular, one can define an $\AR$-closure of a set $E$ in $\R^n$ as the intersection of all $\AR$-closed sets in $\R^n$ which contain $E$.

For a semialgebraic set $E$ in $\R^n$, let $\overline{E}^{\Zar}$ denote the Zariski closure of $E$, that is, the smallest real-algebraic subset of $\R^n$ containing $E$. Similarly, let $\overline{E}^{\AR}$ denote the $\AR$-closure of $E$ in $\R^n$. Consider the following three kinds of dimension of $E$:
\begin{itemize}
\item the geometric dimension $\dim_{\mathrm{g}}\!E$, defined as the maximum dimension of a real-analytic submanifold of (an open subset of) $\R^n$ contained in $E$,
\item the algebraic dimension $\dim_a\!E$, defined as $\dim\overline{E}^{\Zar}$,
\item the $\AR$ topological (or Krull) dimension $\dim_\mathrm{K}\!E$, defined as the maximum length $l$ of a chain $X_0\subsetneq X_1\subsetneq\dots\subsetneq X_l\subset\overline{E}^{\AR}$, where $X_0,\dots,X_l$ are $\AR$-irreducible.
\end{itemize}
It is well known that $\dim_{\mathrm{g}}\!E=\dim_a\!E$ (see, e.g., \cite[Sec.\,2.8]{BCR}). By \cite[Prop.\,2.11]{Ku}, we also have $\dim_a\!E=\dim_\mathrm{K}\!E$.
We shall denote this common dimension simply as $\dim{E}$. By convention, $\dim\varnothing=-1$.

\subsection{Blowings-up and desingularization}

An essential tool in the proof of Theorem~\ref{thm:main} is the blowing-up of $\R^n$ at a Nash subset. Recall that a subset $Z$ of a semialgebraic open $U\subset\R^n$ is called \emph{Nash} if it is the zero locus of a Nash function $f:U \to \R$. A function $f:U\to\R$ is called a \emph{Nash function} if it is an analytic algebraic function on $U$, that is, a real-analytic function such that there exists a non-zero polynomial $P\in\R[x,t]$ with $P(x,f(x))=0$, for every $x \in U$. We denote the ring of all Nash functions on $U$  by $\Na(U)$. We refer the reader to \cite[Ch.\,8]{BCR} for details on Nash sets and mappings.

Let $Z$ be a Nash subset of $\R^n$. Consider the ideal $\I(Z)$ in $\Na(\R^n)$ of all Nash functions on $\R^n$ vanishing on $Z$. By noetherianity of $\Na(\R^n)$ (see, e.g., \cite[Thm.\,8.7.18]{BCR}), there are $f_1,\dots,f_r \in \Na(\R^n)$ such that $\I(Z)=(f_1,\dots,f_r)$. Set
\[
\widetilde{R} \coloneqq \{(x,[u_1,\dots,u_r])\in\R^n\times\mathbb{RP}^{r-1} : \ u_if_j(x)=u_jf_i(x) \mathrm{\ for\ all\ }i,j=1,\dots,r\}\,.
\]
The restriction $\sigma:\widetilde{R}\to\R^n$ to $\widetilde{R}$ of the canonical projection $\R^n\times\mathbb{RP}^{r-1}\to\R^n$ is the \emph{blowing-up of $\R^n$ at (the centre) $Z$}. One can verify that $\widetilde{R}$ is independent of the choice of generators $f_1,\dots,f_r$ of $\I(Z)$.
Since a real projective space is an affine algebraic set (see, e.g., \cite[Thm.\,3.4.4]{BCR}), one can  assume that $\widetilde{R}$ is a Nash subset of $\R^N$ for some $N\in\N$. If $X$ is a Nash subset of $\R^n$, then the smallest Nash subset $\widetilde{X}$ of $\widetilde{R}$ containing $\sigma^{-1}(X\setminus Z)$ is called the \emph{strict transform of $X$ (by $\sigma$)}. In this case, if $Z\subset X$, then we may also call $\widetilde{X}$ the \emph{blowing-up of $X$ at $Z$}.
\medskip

For a semialgebraic set $E$ and a natural number $d$, we denote by $\Reg_d(E)$ the semialgebraic set of those points $x\in E$ at which $E_x$ is a germ of a $d$-dimensional analytic manifold. If $\dim{E}=k$, then $\dim(E\setminus\Reg_k(E))<\dim{E}$.

For a real algebraic set $X$, we denote by $\Sing(X)$ the singular locus of $X$ in the sense of \cite[\S\,3.3]{BCR}. Then, $\Sing(X)$ is an algebraic set of dimension strictly less than $\dim{X}$. Note that, in general, we may have $\Sing(X)\supsetneq X\setminus\Reg_k(X)$, where $k=\dim{X}$.

Recall that every algebraic set $X$ in $\R^n$ admits an \emph{embedded desingularization}. That is, there exists a proper mapping $\pi:\widetilde{R}\to\R^n$ which is the composition of a finite sequence of blowings-up with smooth algebraic centres, such that $\pi$ is an isomorphism outside the preimage of the singular locus $\Sing(X)$ of $X$, the strict transform $\widetilde{X}$ of $X$ is smooth, and $\widetilde{X}$ and $\pi^{-1}(\Sing(X))$ simultaneously have only normal crossings. (The latter means that every point of $\widetilde{R}$ admits a (local analytic) coordinate neighbourhood in which $\widetilde{X}$ is a coordinate subspace and each hypersurface $H$ of $\pi^{-1}(\Sing(X))$ is a coordinate hypersurface.)
For details on resolution of singularities we refer the reader to \cite{BM2} or \cite{Hi}.

\subsection{Nash functions on monomial singularities}

Another key component in the proof of Theorem~\ref{thm:main} is the behaviour of Nash functions on the so-called monomial singularities, studied in \cite{BFR}. Let $M\subset\R^n$ be an \emph{affine Nash submanifold}, that is, a semialgebraic set which is a closed real analytic submanifold of an open set in $\R^n$. Let $X\subset M$ and let $\xi\in X$. We say that the germ $X_\xi$ is a \emph{monomial singularity} if there is a neighbourhood $U$ of $\xi$ in $M$ and a Nash diffeomorphism $u:U\to\R^m$, with $u(\xi)=0$, that maps $X\cap U$ onto a union of coordinate subspaces. We say that $X$ is a \emph{set with monomial singularities} if its germ at every point is a monomial singularity (possibly smooth).

Given a semialgebraic subset $E$ of an affine Nash submanifold $M$, a function $f:E\to\R$ is called a \emph{Nash function on $E$} if there exists an open semialgebraic $U$ in $M$, with $E\subset U$, and a Nash function $F\in\Na(U)$ (in the sense defined above) such that $F|_E=f$. The ring of all Nash functions on $E$ will be denoted by $\Na(E)$. If $E\subset M$ is a Nash set, then a function $f:E\to\R$ is called a \emph{c-Nash function} when its restriction to each irreducible component of $E$ is Nash. The ring of c-Nash functions will be denoted by $\!\!{~}^c\!\Na(E)$. Of course, we always have $\Na(E)\subset\!\!{~}^c\!\Na(E)$. By \cite[Thm.\,1.6]{BFR}, if $E\subset M$ is a Nash set with monomial singularities then
\begin{equation}
\label{eq:BFR}
\Na(E)=\!\!{~}^c\!\Na(E)\,.
\end{equation}


\section{Proof of Theorem~\ref{thm:main}}
\label{sec:proof}

For a semialgebraic set $S$ in $\R^n$ and an integer $k$, we will denote by $\Reg_{<k}(S)$ the semialgebraic set of these points $x\in S$ at which $S_x$ is a germ of a manifold of dimension less than $k$.

\begin{lemma}
\label{lem:complement}
If $X$ is an $\AR$-closed set of dimension $k$ in $\R^n$, then
\[
X\cap\overline{\overline{X}^{\Zar}\setminus X} \ \subset \ \Sing(\overline{X}^{\Zar})\cup\overline{\Reg_{<k}(X)}\,.
\]
\end{lemma}

\begin{proof}
Set $\Sing_k(X)\coloneqq\overline{\Reg_k(X)}\setminus\Reg_k(X)$. Then $X$ can be written as a union
\[
X=\Reg_k(X)\cup(\Sing_k(X)\cup\overline{\Reg_{<k}(X)})\,.
\]
It is evident that $\Reg_k(X)\cap\overline{\overline{X}^{\Zar}\setminus X}\ \subset\ \overline{X}^{\Zar}\setminus\Reg_k(\overline{X}^{\Zar})$, and hence
\[
\Reg_k(X)\cap\overline{\overline{X}^{\Zar}\setminus X}\ \subset \ \Sing(\overline{X}^{\Zar})\,.
\]
It thus suffices to show that $\Sing_k(X)\subset\Sing(\overline{X}^{\Zar})$.
Suppose otherwise, and pick $\xi\in\Sing_k(X)\cap\Reg_k(\overline{X}^{\Zar})$. Let $U$ be the connected component of $\Reg_k(\overline{X}^{\Zar})$ that contains $\xi$. Then, $U\cap X$ is a non-empty open subset of $X$. On the other hand, $U\setminus X\neq\varnothing$, for else $X_\xi$ would be a smooth $k$-dimensional germ. Pick any $a\in U\cap X$ and $b\in U\setminus X$, and let $\gamma:(-1,1)\to U$ be an analytic arc in $U$ passing through $a$ and $b$ (which exists, because $U$ is a connected analytic manifold). Then $\gamma^{-1}(X)$ contains a non-empty open subset
of $(-1,1)$, but $\gamma((-1,1))\not\subset X$, which contradicts the arc-symmetry of $X$.
\qed
\end{proof}

\subsubsection*{Proof of Theorem~\ref{thm:main}}
Let $X$ be an $\AR$-closed set in $\R^n$. We argue by induction on dimension of $X$.

If $\dim X=0$, then $X$ is just a finite set and hence an extension $F:\R^n\to\R$ may be even chosen to be polynomial.
Suppose then that $\dim X=k>0$, and every arc-analytic semialgebraic function on every $\AR$-closed set in $\R^n$ of dimension smaller than $k$ admits an arc-analytic semialgebraic extension to the whole $\R^n$.

Given $f\in\A(X)$, let $S(f)$ denote the locus of points $x\in\Reg_k(X)$ such that $f$ is not analytic at $x$. Then, $S(f)$ is semialgebraic and $\dim S(f)\leq k-2$ (see \cite{KuPa}, and cf. \cite[Thm.\,5.2]{Ku}).

Let
\[
Z:=\Sing(\overline{X}^\Zar)\cup\;\overline{S(f)\cup\Reg_{<k}(X)}^{\Zar}\,.
\]
Since taking Zariski closure of a semialgebraic set does not increase the dimension, we have $\dim(Z\cap X)\leq k-1$.
Therefore, by the inductive hypothesis, $f|_{Z\cap X}$ can be extended to an arc-analytic semialgebraic function $g:\R^n\to\R$. By replacing $f$ with $f-g|_X$, we may thus assume that
\begin{equation}
\label{eq:Z}
f|_{Z\cap X}=0.
\end{equation}
We may further extend $f$ to an arc-analytic function on $X\cup Z$, by setting $f|_Z\coloneqq0$, and hence extend it by $0$ to $\overline{X}^{\Zar}$: 
\begin{equation}
\label{eq:Zar}
f|_{\overline{X}^{\Zar}\setminus X}\coloneqq0.
\end{equation}
This extension is arc-analytic. Indeed, by Lemma~\ref{lem:complement}, we have $X\cap\overline{\overline{X}^{\Zar}\setminus X}\subset Z$, which, by the arc-symmetry of $X$, implies that an analytic arc $\gamma$ in $\overline{X}^{\Zar}$ is either entirely contained in $X$ or else it intersects $X$ only at points of $Z$.

Let $\pi:\widetilde{R}\to\R^n$ be an embedded desingularization of $\overline{X}^{\Zar}$\!, and let $\widetilde{X}$ be the strict transform of $\overline{X}^{\Zar}$. By \cite[Thm.\,2.6]{Ku}, there are connected components $E_1,\dots,E_s$ of $\widetilde{X}$, each of dimension $k$, such that $\pi(E_1\cup\ldots\cup E_s)=\overline{\Reg_k(X)}$. Set $E:=E_1\cup\dots\cup E_s$.
By \eqref{eq:Z} and \eqref{eq:Zar}, we have $f\circ\pi|_T\equiv0$ for all other connected components $T$ of $\widetilde{X}$, as well as $f\circ\pi|_H\equiv0$ for every hypersurface $H$ of the exceptional locus $\pi^{-1}(\Sing(\overline{X}^\Zar))$.

By \cite[Thm.\,1.1]{BM1}, there exists a finite composition of blowings-up $\sigma:\check{R}\to\widetilde{R}$ \;(with smooth Nash centres) which converts the arc-analytic semialgebraic function $f\circ\pi|_E$ into a Nash function $f\circ\pi\circ\sigma|_{\check{E}}$, where the Nash manifold $\check{E}$ is the strict transform of $E$ by $\sigma$.
Moreover, by \cite[Thm.\,1.3]{KuPa}, the centres of the blowings-up in $\sigma$ can be chosen so that $\sigma$ is an isomorphism outside the preimage of $S(f\circ\pi)$. Consequently, one can assume that $\pi\circ\sigma$ is an isomorphism outside the preimage of $Z$.

Let $W:=(\pi\circ\sigma)^{-1}(\overline{X}^{\Zar})$. By the above, the singular locus of $W$ is contained in $(\pi\circ\sigma)^{-1}(Z)$.
Let $\tau:\widehat{R}\to\check{R}$ be an embedded desingularization of $W$ (with smooth Nash centres), and let $\widehat{W}$ be the strict transform of $W$. Further, let $\widehat{E}$ be the strict transform of $\check{E}$, and let $\Sigma$ denote the exceptional locus of $\tau$. Since the real projective space is an affine algebraic variety, we may assume that $\widehat{R}\subset\R^N$ for some $N\in\N$. Notice that, by \eqref{eq:Z} and \eqref{eq:Zar}, $f\circ\pi\circ\sigma\circ\tau$ is a continuous function on $\widehat{W}\cup\Sigma$, which vanishes identically on every irreducible component of $\widehat{W}$ which is not contained in $\widehat{E}$, as well as on every irreducible component of $\Sigma$. Since $f\circ\pi\circ\sigma\circ\tau|_{\widehat{E}}$ is Nash, by construction, it follows that $f\circ\pi\circ\sigma\circ\tau$ is Nash when restricted to every (Nash) irreducible component of $\widehat{W}\cup\Sigma$. We will write $\widehat{f}$ for $f\circ\pi\circ\sigma\circ\tau|_{\widehat{W}\cup\Sigma}$, for short.

We claim that $\widehat{f}$ can be extended to a Nash function $\widehat{F}:U\to\R$ on an open semialgebraic neighbourhood $U$ of $\widehat{W}\cup\Sigma$ in $\R^N$. Indeed, the set $\widehat{W}\cup\Sigma$ is a finite union of Nash submanifolds of $\R^N$ which simultaneously have only normal crossings. Therefore, by \eqref{eq:BFR}, $\widehat{f}$ admits a required Nash extension if and only if $\widehat{f}|_T$ can be extended to a Nash function on an open semialgebraic neighbourhood of $T$ in $\R^N$ for every irreducible component $T$ of $\widehat{W}\cup\Sigma$. Let then $T$ be such an irreducible component. Since $T$ is a Nash submanifold of $\R^N$, it has a tubular neighbourhood. That is, there exists an open semialgebraic neighbourhood $U_T$ of $T$ in $\R^N$ with a Nash retraction $\varrho_T:U_T\to T$ (see \cite[Cor.\,8.9.5]{BCR}). We may thus extend $\widehat{f}|_T$ to a Nash function $\widehat{F}_T:U_T\to\R$ by setting $\widehat{F}_T(x):=\widehat{f}(\varrho_T(x))$ for all $x\in U_T$. This proves the existence of $\widehat{F}$.

Now, by the Efroymson extension theorem (see \cite{E}, or \cite[Thm.\,8.9.12]{BCR}), the function $\widehat{F}$ admits a Nash extension to the whole $\R^N$; i.e., there exists $G\in\Na(\R^N)$ such that $G|_U=\widehat{F}$. Then, $G|_H\equiv0$ for every hypersurface $H$ of the exceptional locus of $\tau$, since this is the case for $\widehat{F}$.

Finally, we define the extension $F:\R^n\to\R$ of $f$ as
\[
F(x):=\begin{cases}
(G\circ\tau^{-1}\circ\sigma^{-1}\circ\pi^{-1})(x) & \mathrm{if\ }x\notin Z\\
0 & \mathrm{if\ }x\in Z\ .
\end{cases}
\]
To see that $F$ is arc-analytic, let $\gamma:(-1,1)\to\R^n$ be an analytic arc. Let $\widetilde{\gamma}:(-1,1)\to\widetilde{R}$ be the lifting of $\gamma$ by $\pi$, let $\check{\gamma}:(-1,1)\to\check{R}$ be the lifting of $\widetilde{\gamma}$ by $\sigma$, and let $\widehat{\gamma}:(-1,1)\to\R^N$ be the lifting of $\check{\gamma}$ by $\tau$. We claim that
\begin{equation}
\label{eq:arc}
F\circ\gamma=G\circ\widehat{\gamma}\,,
\end{equation}
which implies that $F\circ\gamma$ is analytic.
Indeed, if $\gamma(t)\notin Z$ for some $t\in(-1,1)$, then \eqref{eq:arc} holds because $(G\circ\tau^{-1}\circ\sigma^{-1}\circ\pi^{-1})(\gamma(t))=(G\circ\tau^{-1}\circ\sigma^{-1})(\widetilde{\gamma}(t))=(G\circ\tau^{-1})(\check{\gamma}(t))=G(\widehat{\gamma}(t))$. If, in turn, $\gamma(t)\in Z$, then $\gamma(t)$ lifts by $\pi\circ\sigma\circ\tau$ either to a point $z$ in $\widehat{W}\setminus\widehat{E}$ or else a point $z$ in the exceptional locus of $\tau$. In either case, $G(z)=0$, by construction, and so $G(\widehat{\gamma}(t))=0=F(\gamma(t))$, as required.
\qed

\begin{remark}
\label{rem:where-sing}
It is evident from the above proof that, in fact, one could choose the extension $F:\R^n\to\R$ to be analytic outside of \,$\Sing(\overline{X}^\Zar)\cup\overline{S(f)\cup\Reg_{<k}(X)}^{\Zar}$ (hence, in particular, outside of $\overline{X}^\Zar$).
\end{remark}

\begin{problem}
\label{prob:1}
It would be interesting to know if the extension $F$ can be chosen so that its non-analyticity locus satisfies $S(F)=S(f)$.
\end{problem}


\section{Some immediate applications}
\label{sec:corollaries}

Arc-analytic semialgebraic functions may be defined and studied on arbitrary semialgebraic sets (see, e.g., \cite{S}). It is thus natural to ask which semialgebraic sets enjoy the extension property from Theorem~\ref{thm:main}. The following result shows that, in fact, the arc-symmetric sets are uniquely characterised by the extension property.

\begin{proposition}
\label{prop:only-AR}
For a semialgebraic set $S$ in $\R^n$, the following conditions are equivalent:
\begin{itemize}
\item[(i)] $S$ is arc-symmetric.
\item[(ii)] Every arc-analytic semialgebraic function on $S$ admits an arc-analytic semialgebraic extension to the whole $\R^n$.
\end{itemize}
\end{proposition}

\begin{proof}
The implication $\mathrm{(i)}\Rightarrow\mathrm{(ii)}$ is given by Theorem~\ref{thm:main}.
For the converse, let $S$ be a semialgebraic subset of $\R^n$ that is not arc-symmetric. This means that there exists an analytic arc $\gamma:(-1,1)\to\R^n$ such that $\gamma((-1,0))\subset S$ but $\gamma((0,1))\not\subset S$. Pick a point $a=(a_1,\dots,a_n)\in\gamma((0,1))\setminus S$, and define
\[
f(x)=\frac{1}{\sum_{i=1}^n (x_i-a_i)^2}\,,\quad\mathrm{where\ }x=(x_1,\dots,x_n)\in\R^n\,. 
\]
Then $f$ is an arc-analytic function on $S$ that has no extension to an arc-analytic function on $\R^n$. Indeed, given any such extension $F:\R^n\to\R$, we would have $F(\gamma(t))=f(\gamma(t))$ for any $t \in (-1,0)$ and hence for any $t \in (-1,s)$, where $s\in(0,1)$ is the minimum parameter such that $\gamma(s)=a$. But $f(\gamma(t))$ has no left-sided limit at $s$, which means that $F\circ\gamma$ cannot be analytic, and thus $F$ is not arc-analytic.
\qed
\end{proof}

Theorem~\ref{thm:main} implies also an arc-analytic variant of the Urysohn lemma. More precisely, we have the following.

\begin{corollary}
\label{cor:TU}
Let $X$ and $Y$ be disjoint $\AR$-closed sets in $\R^n$. Then, there exists an arc-analytic semialgebraic function $F:\R^n\to\R$ such that $F|_X\equiv0$ and $F|_Y\equiv1$. In particular, there exist disjoint open semialgebraic sets $U$ and $V$ in $\R^n$ such that $X\subset U$ and $Y\subset V$.
\end{corollary}

\begin{proof}
Given $X$ and $Y$ as above, the function $f:X\cup Y\to\R$ defined as
\[
f(x)=\begin{cases}0, &x\in X\\ 1, & x\in Y\end{cases}
\]
is arc-analytic semialgebraic, and the set $X\cup Y$ is arc-symmetric. Hence, by Theorem~\ref{thm:main}, $f$ admits an extension $F:\R^n\to\R$ with the required properties.
Since arc-analytic semialgebraic functions are continuous (\cite[Prop.\,5.1]{Ku}), the sets $U:=F^{-1}((-\infty,1/2))$ and $V:=F^{-1}((1/2,\infty))$ are open semialgebraic. Clearly, $U\cap V=\varnothing$, $X\subset U$, and $Y\subset V$.
\qed
\end{proof}

\begin{remark}
\label{rem:no-Nash-sep}
Note that, in general, disjoint arc-symmetric sets cannot be separated by a Nash function. Indeed, consider for instance
\[
X=\{(x,y,z)\in\R^3:\, z(x^2+y^2)=x^3\}\setminus\{(x,y,z)\in\R^3:\, x=y=0,z\neq0\}
\]
and $Y=\{(0,0,1)\}$ in $\R^3$. The set $X$ is $\AR$-closed, but its real analytic closure in $\R^3$ is the irreducible algebraic hypersurface $Z=\{(x,y,z)\in\R^3:\, z(x^2+y^2)=x^3\}$ (see \cite[Ex.\,1.2(1)]{Ku}). It follows that every Nash function $f:\R^3\to\R$ which is identically zero on $X$ must vanish on the whole $Z$ and thus cannot be equal to $1$ on $Y$.

Similarly, it is easy to construct disjoint $\AR$-closed sets that cannot be separated by a continuous rational function (cf. \cite[Ex.\,2.3]{S}).
\end{remark}


\end{document}